\pdfoutput=1
\documentclass[10pt,oneside]{amsart}
\usepackage[
paper=a4paper,
headsep=20pt,text={136mm,208mm},centering,includehead
]{geometry}
\usepackage{graphicx}
\usepackage[dvipsnames,usenames]{xcolor}
\usepackage[colorlinks=true,urlcolor=ForestGreen,linkcolor=ForestGreen,citecolor=ForestGreen]{hyperref}
\hypersetup{
	linkbordercolor={1 0 0}, % set to white
	citebordercolor={0 1 0} % set to white
}
\usepackage{cite}
\usepackage{wasysym}
\usepackage{amssymb,mathtools,mathrsfs,stmaryrd,mwe}
\usepackage{bbm}
\usepackage[UKenglish]{babel}
\usepackage[noabbrev]{cleveref}
\usepackage{tikz}
\usetikzlibrary{trees,decorations.pathmorphing,decorations.markings,matrix,shapes}
\usepackage[backgroundcolor=green,linecolor=green,bordercolor=green]{todonotes}
\usepackage{enumitem}
\usepackage{soul} %For strikethrough
\setstcolor{red}

\usepackage{libertine}
\usepackage[vvarbb,bigdelims,cmintegrals,libertine]{newtxmath}

\iftrue
\makeatletter
\def\@settitle{%
	\vspace*{-20pt}
	\begin{flushleft}%
		\baselineskip14\p@\relax
		\normalfont\bfseries\LARGE
		\@title
	\end{flushleft}%
}
\def\@setauthors{%
	\begingroup
	\def\thanks{\protect\thanks@warning}%
	\trivlist
	\large \@topsep30\p@\relax
	\advance\@topsep by -\baselineskip
	\item\relax
	\author@andify\authors
	\def\\{\protect\linebreak}%
	\authors
	\ifx\@empty\contribs
	\else
	,\penalty-3 \space \@setcontribs
	\@closetoccontribs
	\fi
	\normalfont
	\@setaddresses
	\endtrivlist
	\endgroup
}
\def\@setaddresses{\par
	\nobreak \begingroup\raggedright
	\small
	\def\author##1{\nobreak\addvspace\smallskipamount}%
	\def\\{\unskip, \ignorespaces}%
	\interlinepenalty\@M
	\def\address##1##2{\begingroup
		\par\addvspace\bigskipamount\noindent
		\@ifnotempty{##1}{(\ignorespaces##1\unskip) }%
		{\ignorespaces##2}\par\endgroup}%
	\def\curraddr##1##2{\begingroup
		\@ifnotempty{##2}{\nobreak\noindent\curraddrname
			\@ifnotempty{##1}{, \ignorespaces##1\unskip}\/:\space
			##2\par}\endgroup}%
	\def\email##1##2{\begingroup
		\@ifnotempty{##2}{\smallskip\nobreak\noindent %
			\@ifnotempty{##1}{, \ignorespaces##1\unskip}
			\ttfamily##2\par}\endgroup}%
	\def\urladdr##1##2{\begingroup
		\def~{\char`\~}%
		\@ifnotempty{##2}{\nobreak\noindent\urladdrname
			\@ifnotempty{##1}{, \ignorespaces##1\unskip}\/:\space
			\ttfamily##2\par}\endgroup}%
	\addresses
	\endgroup
	\global\let\addresses=\@empty
}
\def\@setabstracta{%
	\ifvoid\abstractbox
	\else
	\skip@25\p@ \advance\skip@-\lastskip
	\advance\skip@-\baselineskip \vskip\skip@
	\box\abstractbox
	\prevdepth\z@ % because \abstractbox is a vtop
	\vskip-15pt
	\fi
}
\renewenvironment{abstract}{%
	\ifx\maketitle\relax
	\ClassWarning{\@classname}{Abstract should precede
		\protect\maketitle\space in AMS document classes; reported}%
	\fi
	\global\setbox\abstractbox=\vtop \bgroup
	\normalfont\small
	\list{}{\labelwidth\z@
		\leftmargin0pc \rightmargin\leftmargin
		\listparindent\normalparindent \itemindent\z@
		\parsep\z@ \@plus\p@
		
	}%
	\item[\hskip\labelsep\bfseries\abstractname.]%
}{%
	\endlist\egroup
	\ifx\@setabstract\relax \@setabstracta \fi
}

\def\ps@headings{\ps@empty
	\def\@evenhead{%
		\setTrue{runhead}%
		\normalfont\scriptsize
		\rlap{\thepage}\hfill
		\def\thanks{\protect\thanks@warning}%
		\leftmark{}{}}%
	\def\@oddhead{%
		\setTrue{runhead}%
		\normalfont\scriptsize
		\def\thanks{\protect\thanks@warning}%
		\rightmark{}{}\hfill \llap{\thepage}}%
	\let\@mkboth\markboth
}\ps@headings

\def\section{\@startsection{section}{1}%
	\z@{-1.2\linespacing\@plus-.5\linespacing}{.8\linespacing}%
	{\normalfont\bfseries\Large}}
\def\subsection{\@startsection{subsection}{2}%
	\z@{-.8\linespacing\@plus-.3\linespacing}{.3\linespacing\@plus.2\linespacing}%
	{\normalfont\bfseries\large}}
\def\subsubsection{\@startsection{subsubsection}{3}%
	\z@{.7\linespacing\@plus.1\linespacing}{-1.5ex}%
	{\normalfont\bfseries}}
\def\@secnumfont{\bfseries}
\makeatother
\fi %\iftrue/false for amsart.cls hack

\makeatother

\newtheorem{theorem}{Theorem}
\newtheorem*{theorem*}{Theorem}
\newtheorem*{corollary*}{}
\newtheorem{proposition}[theorem]{Proposition}
\newtheorem{corollary}[theorem]{Corollary}

\newtheorem*{question}{Question}
\newtheorem{example}[theorem]{Example}

\theoremstyle{definition}
\newtheorem{defin}[theorem]{Definition}
\newenvironment{definition}
{\pushQED{\qed}\defin}
{\popQED\enddefin}

\makeatletter

\providecommand{\proofname}{Proof}
\makeatother

\newcommand{\Z}{\mathbb{Z}}

\begin{document}
	
\vspace*{-40pt}
\title{Minimal crossing number implies minimal supporting genus}
	
\author{Hans U.\ Boden}
\address{
Mathematics \& Statistics, McMaster University, Hamilton, Ontario
}
\email{\href{mailto:boden@mcmaster.ca}{boden@mcmaster.ca}, \href{mailto:will.rushworth@math.mcmaster.ca}{will.rushworth@math.mcmaster.ca}}

\author{William Rushworth}
%\address{
%}
%\email{\href{mailto:will.rushworth@math.mcmaster.ca}{will.rushworth@math.mcmaster.ca}}
	
\def\subjclassname{\textup{2020} Mathematics Subject Classification}
\expandafter\let\csname subjclassname@1991\endcsname=\subjclassname
\expandafter\let\csname subjclassname@2000\endcsname=\subjclassname
	
	
\begin{abstract}
	A virtual link may be defined as an equivalence class of diagrams, or alternatively as a stable equivalence class of links in thickened surfaces. We prove that a minimal crossing virtual link diagram has minimal genus across representatives of the stable equivalence class. This is achieved by constructing a new parity theory for virtual links. As corollaries, we prove that the crossing, bridge, and ascending numbers of a classical link do not decrease when it is regarded as a virtual link. This extends corresponding results in the case of virtual knots due to Manturov and Chernov.
\end{abstract}
	
\maketitle

\section{Introduction}\label{Sec:intro}
In this note we prove that a minimal crossing diagram of a virtual link minimises the genus of surfaces \( \Sigma \) such that the link possesses a representative in \( \Sigma \times I \). This extends the corresponding result in the case of virtual knots due to Manturov \cite{Manturov2010,Manturov13}. Our result affirmatively answers a basic question, open since the inception of virtual knot theory: is it possible to simultaneously minimize the complexity of a diagram and the complexity of the surface supporting it? In fact, we establish a stronger result: a minimal crossing diagram is automatically of minimal genus. 

A virtual link is an equivalence class of virtual link diagrams, up to the generalised Reidemeister moves \cite[Section 2]{Kauffman1998}. The \emph{classical crossing number} of a virtual link is the minimal number of classical crossings, taken over all diagrams of the link.

Equivalently to the diagrammatic formulation, a virtual link may be defined as an equivalence class of smooth embeddings \( \bigsqcup S^1 \hookrightarrow \Sigma \times I \), for \( \Sigma \) a closed orientable surface, up to self-diffeomorphism and (de)stabilization of \( \Sigma \times I \) \cite[Section 3]{Carter2000,Kauffman1998}. This last operation is the addition or removal of a \(1\)-handle of \( \Sigma \) disjoint to the embedding. The \emph{supporting genus} of a virtual link is the minimal genus of a surface \( \Sigma \) such that the link possess a representative in \( \Sigma \times I \).

Let \( D \) be a diagram of the virtual link \( L \). As described in \Cref{Sec:apps}, to \( D \) there is a naturally associated representative of \( L \) in \( \Sigma \times I \); the particular surface \( \Sigma \) produced in this way is known as the \emph{Carter surface of \( D \)}. A diagram of a virtual link \( L \) is said to be \emph{minimal genus} if its Carter surface has genus equal to the supporting genus of \( L \).

\begin{theorem}\label{Thm:main}
	Let \( D \) be a virtual link diagram. If \( D \) is of minimal classical crossing number then it is a minimal genus diagram.
\end{theorem}
That is, a diagram that realises the classical crossing number also realises the supporting genus. This verifies \cite[Conjecture 5.1]{Boden-Karimi-2019} and yields the following result for classical links.

\begin{corollary}\label{Cor:classical}
	The crossing number of a classical link does not decrease when it is considered as a virtual link.
\end{corollary}

\Cref{Thm:main} is proved by introducing a new parity theory for virtual links, before employing a parity projection argument. Parity is a powerful concept in virtual knot theory, and extending it to virtual links is an important task (see \cite[Section 1]{Rushworth2019}). The parity theory for links that we introduce is applicable to a restricted class of virtual link diagrams, and it is the natural extension of the so-called \emph{homological parity} for virtual knots due to Manturov \cite{Manturov2010,Manturov13}. In addition to a combinatorial definition, we present a topological definition of our parity theory in \Cref{Sec:top}.

\Cref{Thm:main} is a consequence of \Cref{Thm:sub}. The latter result guarantees that given a diagram of a virtual link \( L \), one may convert classical crossings to virtual crossings to produce a minimal genus diagram of \( L \). As examples of its utility, we apply \Cref{Thm:sub} to show that the bridge and ascending numbers of a virtual link are realised on minimal genus diagrams (see \Cref{Prop:bridge,Prop:asc}). The corresponding result regarding the bridge number of a virtual knot is due to Chernov \cite{Chernov2013} and Manturov \cite{Manturov13}.

The parity constructed in this note appears to be well-suited to answering questions of the form: given a quantity extracted from a virtual link diagram, can we minimise it on a minimal genus representative? \Cref{Thm:main}, \Cref{Prop:bridge}, and \Cref{Prop:asc} are instances of this question in the case of crossing, bridge, and ascending number, and it is interesting to consider other instances. In particular, we wish to advertise the following open question: can the unknotting number of a virtual knot be realised on a minimal genus diagram?

We construct the requisite parity theory in \Cref{Sec:parity}, before proving \Cref{Thm:main} and related results in \Cref{Sec:apps}.

\subsubsection*{Conventions}
All surfaces are closed and orientable, and are denoted by \( \Sigma \). We denote by \( RI \), \(RII\), and \(RIII\) the classical Reidemeister moves. For our purposes a \emph{link in a thickened surface} is a smooth embedding \( \bigsqcup S^1 \hookrightarrow \Sigma \times I \), considered up to isotopy, where \( I = [0,1] \). A \emph{diagram} of a link in a thickened surface is a link diagram drawn on a surface. Two diagrams of a given link in a thickened surface are related by a finite sequence of the moves \( RI \), \(RII\), \(RIII\), and isotopy, where the Reidemeister moves occur in disc neighbourhoods of \( \Sigma \). We denote diagrams of links in thickened surfaces by the \verb|\mathfrak| character \( \mathfrak{D} \), reserving Roman characters for virtual links and their diagrams.

\subsubsection*{Acknowledgements} We thank Homayun Karimi, Andrew Nicas, and a referee for their helpful comments on an earlier version of this work. We are indebted to Adam Sikora for a number of comments that significantly improved this work. We also thank Zhiyun Cheng, Micah Chrisman, Vassily Manturov, and Puttipong Pongtanapaisan for their valuable feedback.

\section{A homological parity for links}\label{Sec:parity}
In this section we introduce a new theory of parity on virtual links (well-defined for a certain sequences of generalised Reidemeister moves). This parity may be thought of an extension of the homological parity for knots \cite{Manturov2010,Manturov13}, the definition of which is non-local: the full knot diagram is used to determine the parity of a crossing. As a consequence, the construction does not extend to links. The parity that we define in this section is local in nature -- it is computed directly at a crossing -- so that it may naturally be applied to links with arbitrarily many components.

We begin with the combinatorial definition of our parity theory in \Cref{Sec:comb}, before presenting an equivalent topological definition in \Cref{Sec:top}.

\subsection{Combinatorial definition}\label{Sec:comb}
Our construction proceeds as follows. Working at the level of link diagrams on surfaces we introduce a function, \( f_{\mathscr{C}} \), on the crossings of such diagrams. We show that \( f_{\mathscr{C}} \) satisfies the appropriate version of the parity axioms for link diagrams on surfaces, and thus descends to a \emph{bona fide} parity on virtual link diagrams, well-defined for sequences of generalised Reidemeister moves defined by isotopies of links in thickened surfaces.

\begin{figure}
	\includegraphics[scale=0.75]{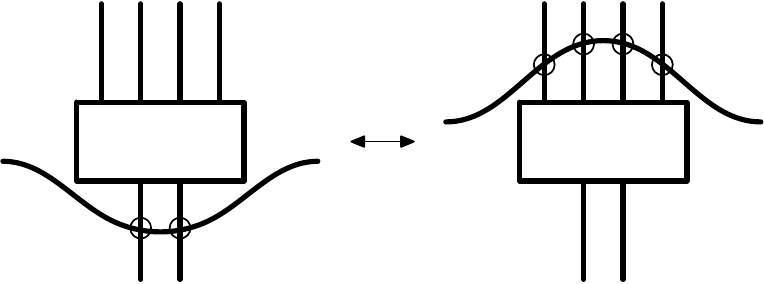}
	\caption{An example of the detour move on virtual link diagrams: a segment containing only virtual crossings may be removed and replaced arbitrarily (with any new crossings produced being virtual).}
	\label{Fig:detour}
\end{figure}

Given an isotopy of links in thickened surfaces there is no canonical way to define a sequence of generalised Reidemeister moves: there is a choice of how some isotopies are realised as detour moves (as in \Cref{Fig:detour}). However, this choice does not affect the resulting parity, nor our proof of \Cref{Thm:main} (and related results). If the reader wishes, they can remove this choice by working \emph{mutatis mutandis} with sequences of Gauss diagrams (in place of virtual link diagrams), as detour moves do not appear in such sequences \cite{GPV}.

First, we state the axioms of a parity for link diagrams on surfaces. These axioms correspond directly to those given by Ilyutko, Manturov, and Nikonov for virtual knots and other knotted objects \cite{Ilyutko2014,Nikonov-2016}.
\begin{definition}
	\label{Def:parityaxioms}
	Consider the category whose objects are link diagrams on surfaces, and morphisms are sequences of Reidemeister moves (where such moves take place on disc neighbourhoods). Given an assignment of a function \( f(\mathfrak{D}) \) to every object \( \mathfrak{D} \), with domain the set of crossings of \( \mathfrak{D} \) and codomain \( \Z_2 \), we refer to the image of a crossing under \( f(\mathfrak{D}) \) as \emph{the parity} of the crossing; crossings that are mapped to \( 0 \) are \emph{even}, and those mapped to \( 1 \) are \emph{odd}.
	Such an assignment of functions is \emph{a parity} if it satisfies the following axioms:
	\begin{enumerate}[start=0]
		\item If diagrams \( \mathfrak{D} \) and \( \mathfrak{D}' \) are related by a single Reidemeister move, then the parities of the crossings that are not involved in this move do not change.
		\item If \( \mathfrak{D}  \) and \( \mathfrak{D} ' \) are related by a Reidemeister I move that eliminates a crossing, then the parity of that crossing is even.
		\item If \( \mathfrak{D}  \) and \( \mathfrak{D} ' \) are related by a Reidemeister II move eliminating the crossings \( c_1 \) and \( c_2 \), then \( c_1 \) and \( c_2 \) are both even or both odd.
		\item If \( \mathfrak{D}  \) and \( \mathfrak{D} ' \) are related by a Reidemeister III move then the parities of the three crossings involved in the move are unchanged. Further, these three crossings are all even, all odd, or exactly two are odd.
		\qedhere
	\end{enumerate}
\end{definition}

An extremely useful application of a parity theory is \emph{parity projection}. Consider a sequence of virtual link diagrams
\begin{equation*}
D_1 \rightarrow D_2 \rightarrow \cdots \rightarrow D_n
\end{equation*}
related by generalised Reidemeister moves. Given a parity, define \( p(D_i) \) to be the diagram obtained from \(D_i \) by replacing every odd crossing with a virtual crossing (that is, \( \raisebox{-4pt}{\includegraphics[scale=0.35]{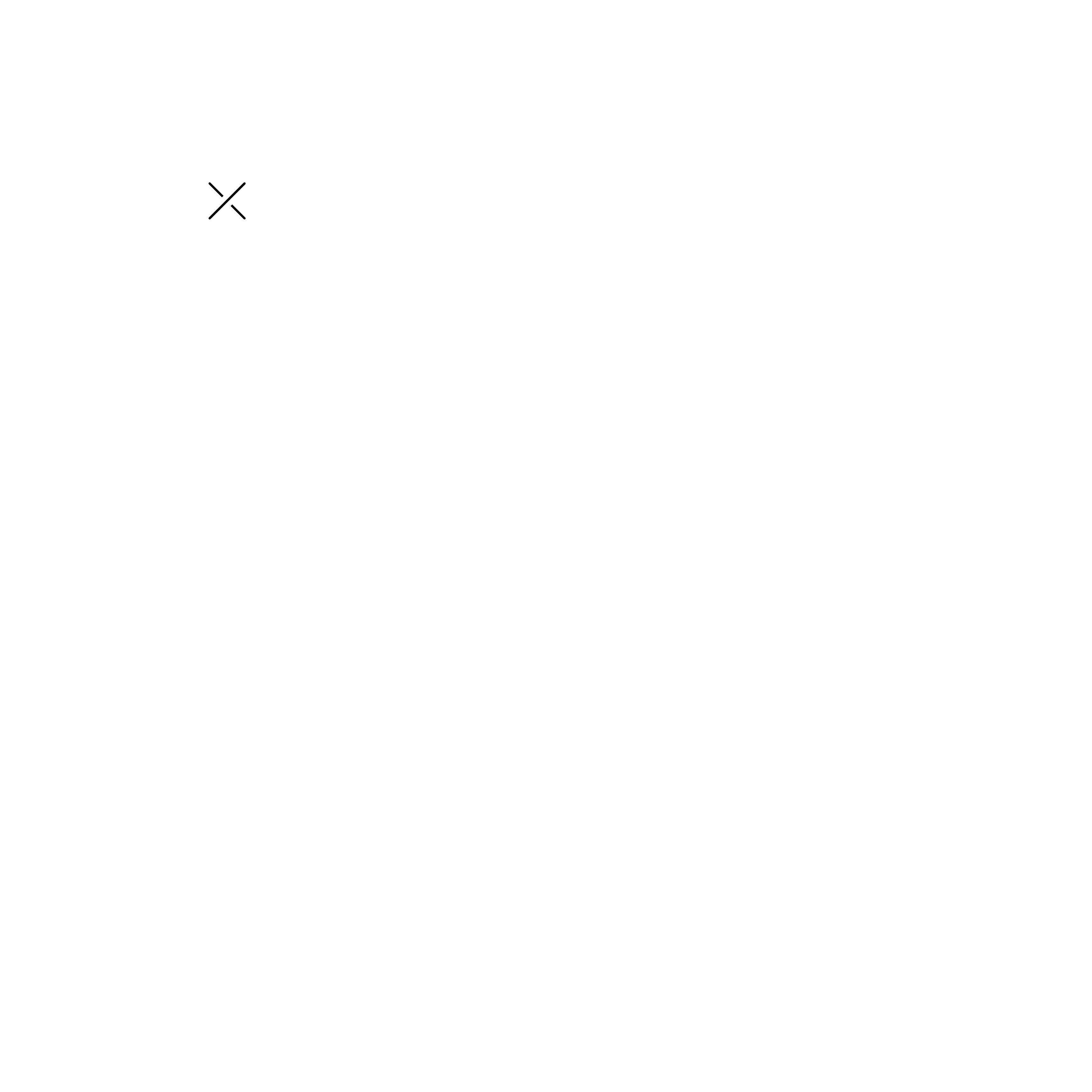}} \to \raisebox{-4pt}{\includegraphics[scale=0.35]{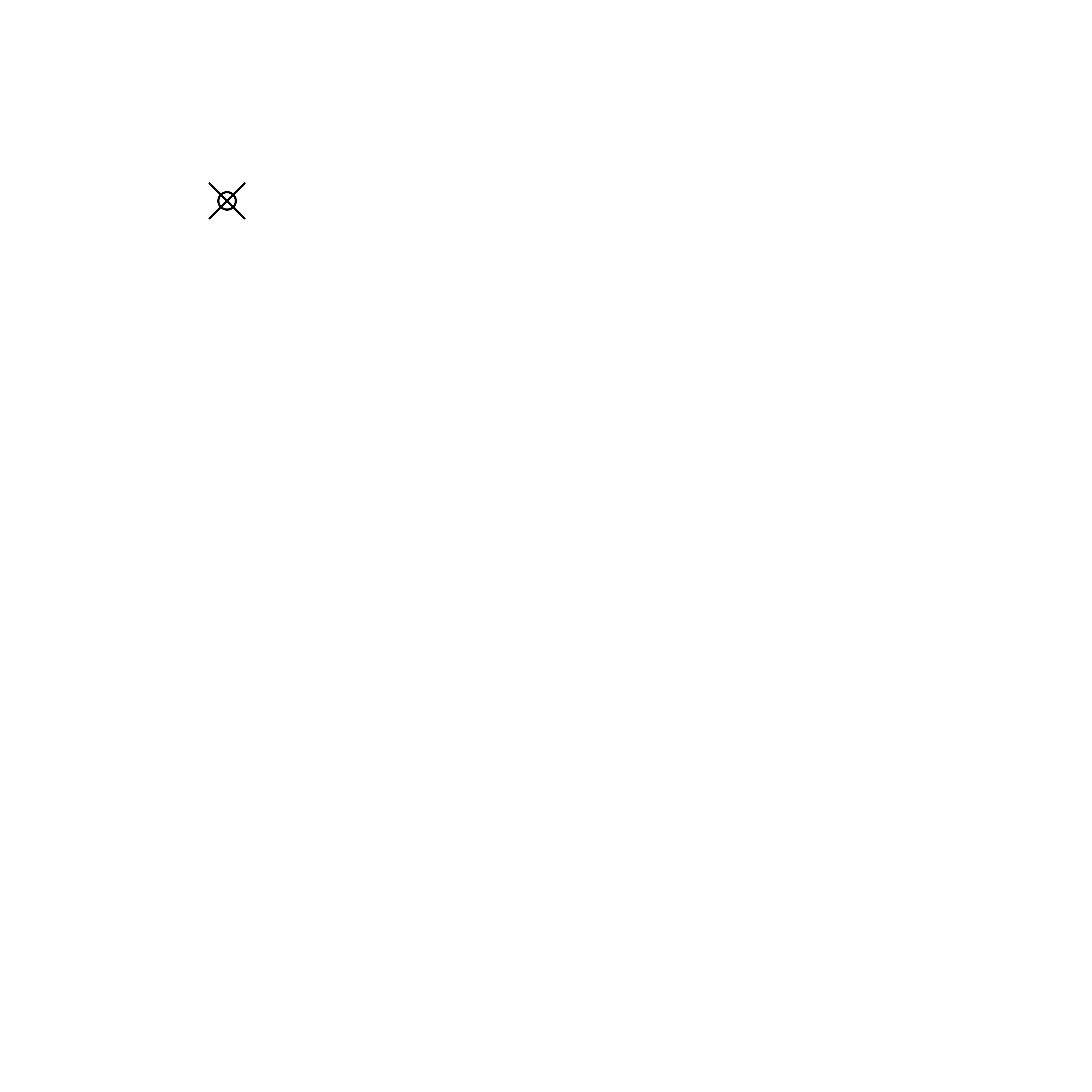}} \)). Virtual links may be defined in terms of Gauss diagrams; on the Gauss diagram of \( D_i \) parity projection corresponds to simply deleting the chords associated to odd crossings. We obtain the new sequence
\begin{equation*}
p(D_1) \rightarrow p(D_2) \rightarrow \cdots \rightarrow p(D_n).
\end{equation*}
That \( p(D_i) \) is related to \( p(D_{i+1}) \) by a generalised Reidemeister move is a consequence of the parity axioms. This operation is known as parity projection; for full details see \cite[Section \(1.3\)]{Manturov13}. We make use of a parity projection argument in the proof of \Cref{Thm:main}.

We use simple closed curves to define colourings of link diagrams on a surface, and use such colourings to define a parity theory.
\begin{definition}[\( \gamma \)-colouring]\label{Def:gcolour}
	Let \( \mathfrak{D} \) be a link diagram on \( \Sigma \), and \( \gamma \subset \Sigma \) a simple closed curve that intersects \( \mathfrak{D} \) transversely and away from crossings, such that every component of \( \mathfrak{D} \) has even intersection number with \( \gamma \). A \emph{\( \gamma \)-colouring of \(\mathfrak{D}\)} is a colouring of the components of \( \mathfrak{D} \) exactly one of two colours, such that the colour switches when passing through \( \gamma \). The colour of a component does not change at a crossing.
\end{definition}
An example of a \( \gamma \)-colouring is given in \Cref{Ex:gcolour}; our convention is to depict the curves \( \gamma \) in red, and use the colours blue and green for the components of link diagrams. Notice that \Cref{Def:gcolour} is well-posed due to the intersection condition. For a fixed curve \( \gamma \), a diagram of a link of \( m \) components possesses either \(0\) or \( 2^m \) \( \gamma \)-colourings; the diagram possesses \(0\) \(\gamma \)-colourings if and only if there is a component that intersects \( \gamma \) an odd number times.
\begin{definition}\label{Def:gparity1}
	Suppose that \( \mathfrak{D} \) and \(\gamma \) are as in \Cref{Def:gcolour}. Let \( \mathscr{C} \) be a \(\gamma \)-colouring of \( \mathfrak{D} \). Define a function on the crossings of \( \mathfrak{D} \), denoted \( f_{\mathscr{C}} \), as follows,
	\begin{equation}\label{Eq:gparity}
		\begin{aligned}
			&f_{\mathscr{C}} \left(\, \raisebox{-11pt}{\includegraphics[scale=0.5]{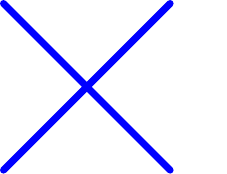}} \! \!\!  \right) = f_{\mathscr{C}} \left(\, \raisebox{-11pt}{\includegraphics[scale=0.5]{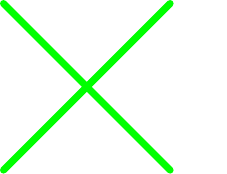}}\! \!\! \right) = 0 \quad \text{and}  
			&f_{\mathscr{C}} \left(\, \raisebox{-11pt}{\includegraphics[scale=0.5]{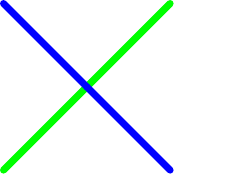}}\! \!\! \right) = 1.
		\end{aligned}
	\end{equation}
\end{definition}

\begin{example}\label{Ex:gcolour}
	A link diagram on a torus, and a \( \gamma \)-colouring of it. With respect to the given \( \gamma \)-colouring two crossings are even, and two are odd.
	\begin{equation*}
	\begin{matrix}
	\includegraphics[scale=0.5]{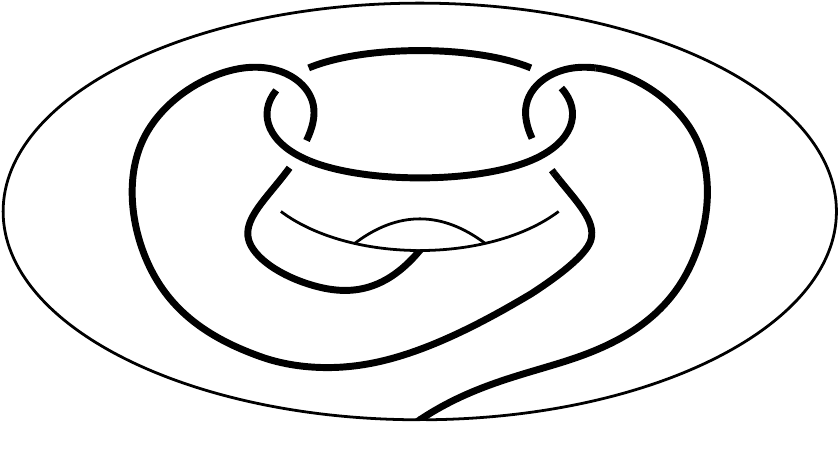} &\qquad &\includegraphics[scale=0.5]{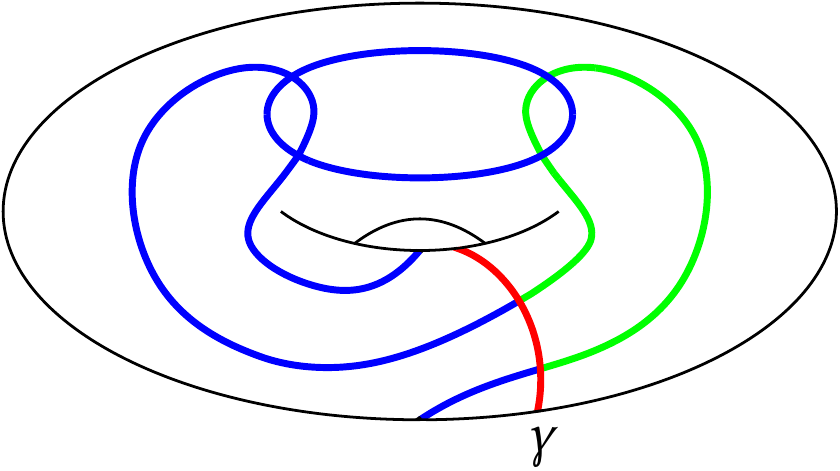} \\
	\end{matrix}
	\end{equation*}
\end{example}

Let \( \mathfrak{D} \), \( \mathfrak{D}' \) be diagrams on \( \Sigma \) related by a single Reidemeister move, and \( \gamma \subset \Sigma \) a simple closed curve. Suppose that \(\mathfrak{D} \) possesses a \(\gamma \)-colouring \( \mathscr{C} \); it is clear that \( \mathscr{C} \) induces a \(\gamma \)-colouring of \( \mathfrak{D}' \).

\begin{proposition}\label{Prop:gparity1}
	Let \( \mathfrak{D} \), \( \gamma \), and \( \mathscr{C} \) be as in \Cref{Def:gparity1}. The function \( f_{\mathscr{C}} \) is a parity.
\end{proposition}

\begin{proof}
	That \( f_{\mathscr{C}} \) satisfies the axioms of \Cref{Def:parityaxioms} may be verified by directly comparing the Reidemeister moves to \Cref{Eq:gparity}, recalling that these moves are supported on disc neighbourhoods of \( \Sigma \). We suffice ourselves with some example verifications. The following hold for any \( \gamma \)-colouring (a possible position of the curve \( \gamma \) is denoted by the red arc):
	\begin{equation*}
			\begin{matrix}
			\includegraphics[scale=0.8]{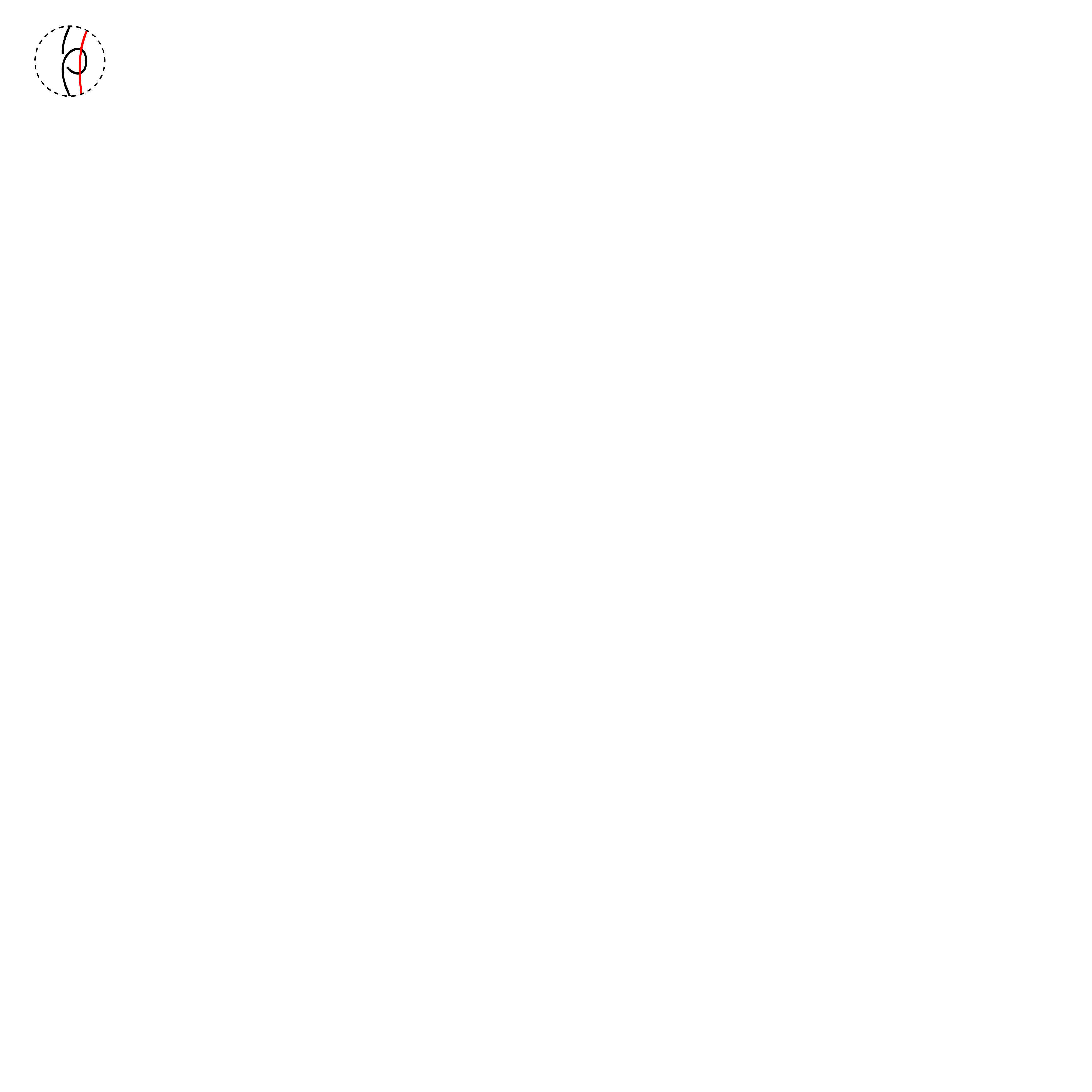} & \quad & \includegraphics[scale=0.8]{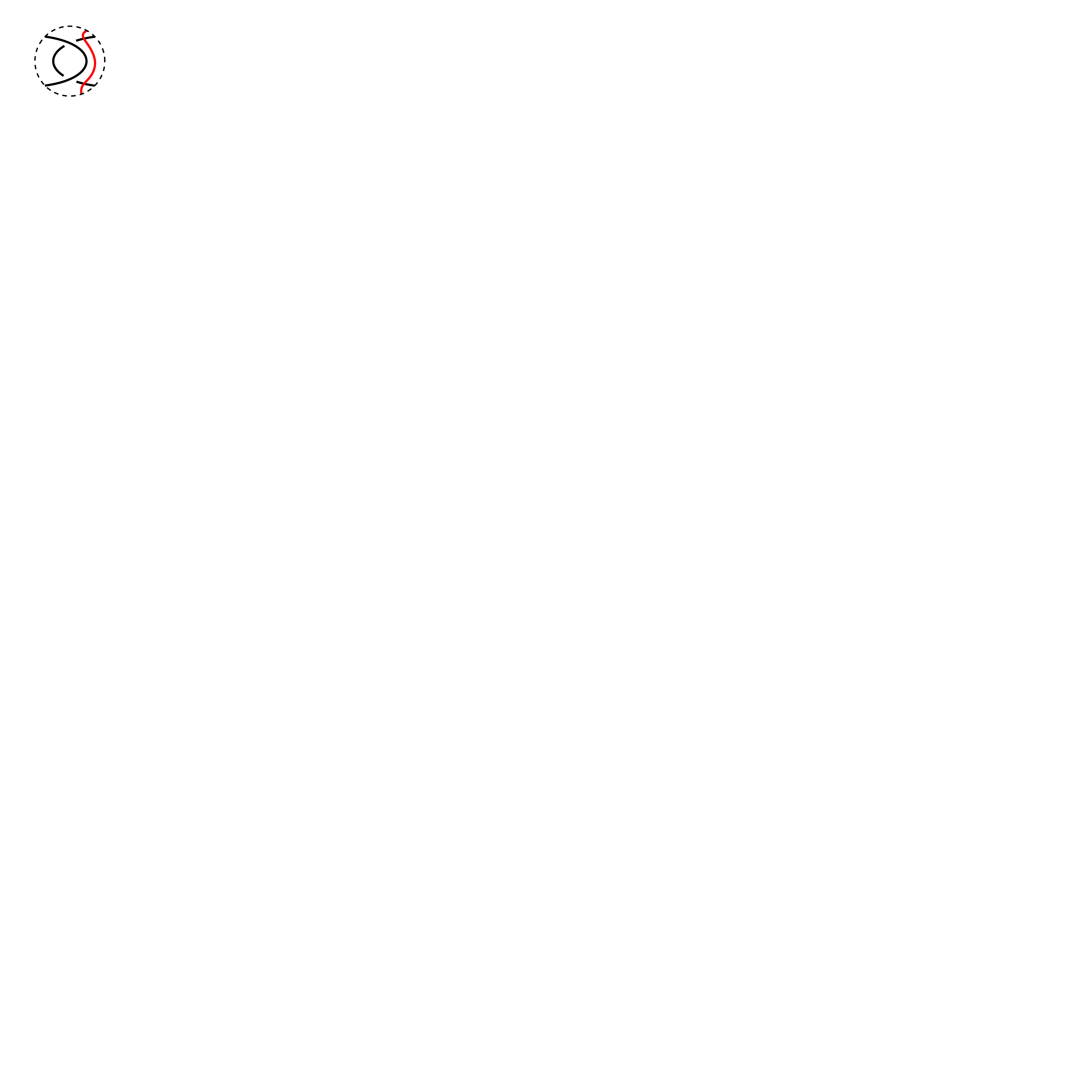} & \quad & \includegraphics[scale=0.8]{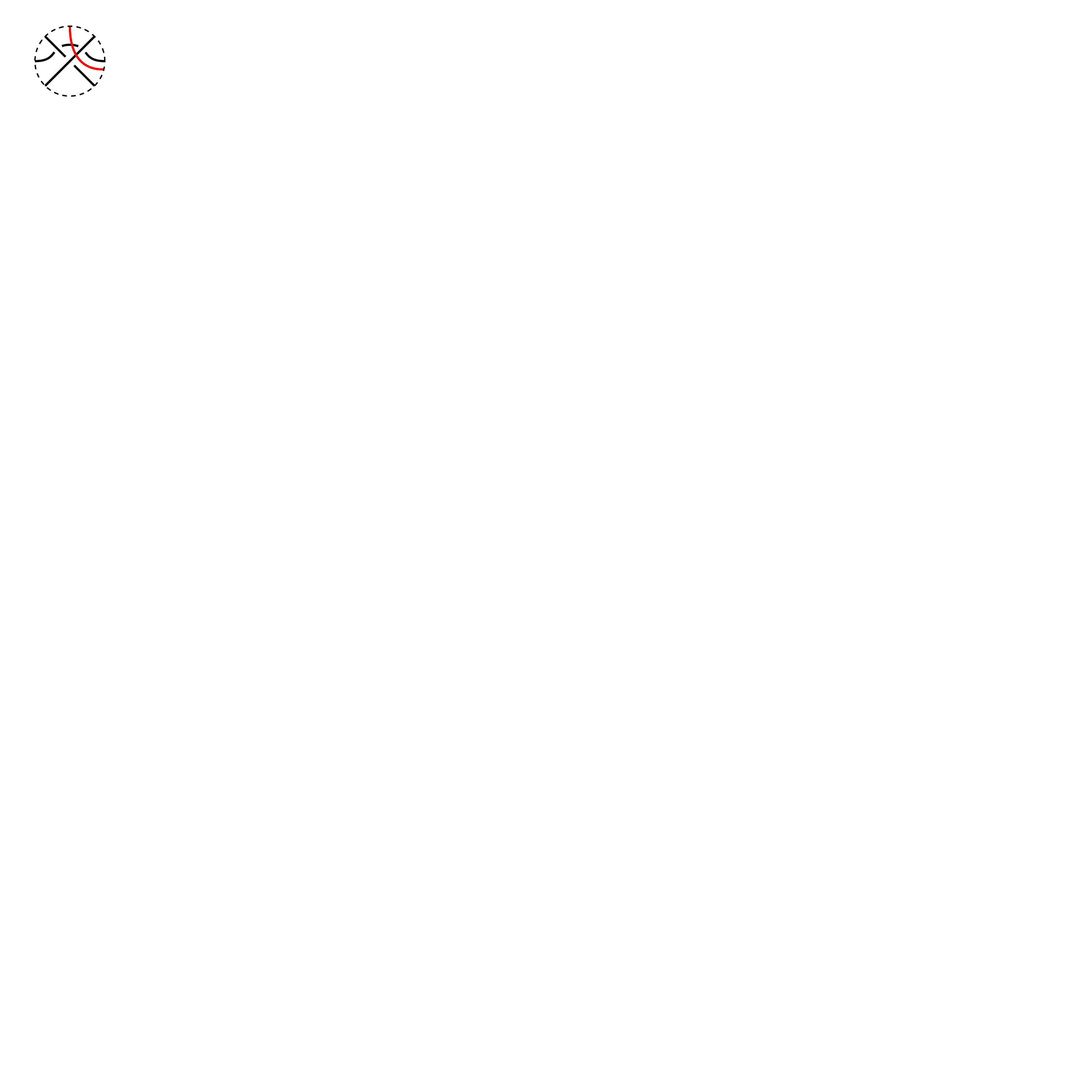} \\
			0~\text{odd crossings} & & 0 ~\text{or}~2~ \text{odd crossings} & & 0 ~\text{or}~2~ \text{odd crossings}
			\end{matrix}
	\end{equation*} 
\end{proof}

Given a sequence of diagrams \( \mathfrak{D}_i \) on \( \Sigma \) such that \( \mathfrak{D}_i \) is related to \( \mathfrak{D}_{i+1} \) by a Reidemeister move, there is a naturally associated sequence of virtual link diagrams \( D_i \), related by generalised Reidemeister moves. Note that the sequence of virtual link diagrams may be longer than that of diagrams on \( \Sigma \): the former sequence may include detour moves, that are not present in the latter. We use \Cref{Prop:gparity1} to define a parity for such (sequences of) virtual link diagrams.

\begin{proposition}\label{Prop:gparity2}
	Let \( \lbrace \mathfrak{D}_i \rbrace \), \( 1 \leq i \leq n \), be a sequence of diagrams on \( \Sigma \) such that \( \mathfrak{D}_i \) is related to \( \mathfrak{D}_{i+1} \) by a Reidemeister move, and \(\lbrace D_j \rbrace \), \( 1 \leq j \leq m \geq n \), the associated sequence of virtual link diagrams. Suppose that \( \mathfrak{D}_i \) possesses a \( \gamma \)-colouring, \( \mathscr{C} \), for some simple closed curve \( \gamma \); by abuse of notation also denote by \( \mathscr{C} \) the induced \( \gamma \)-colouring on \( \mathfrak{D}_j \), \( i \neq j \). Then \( f_{\mathscr{C}} \) descends to a parity on the virtual link diagrams \( D_j \).
\end{proposition}

\begin{proof}
	Suppose that the move \( D_{j} \rightarrow D_{j+1} \) is a detour move. Then there is a one-to-one correspondence between the classical crossings of \( D_{j} \) and \( D_{j+1} \), and both diagrams are associated to the same diagram on \( \Sigma \). It follows that the parity of a crossing is unchanged across the move \( D_{j} \rightarrow D_{j+1} \).
	
	Now suppose that the move \( D_{j} \rightarrow D_{j+1} \) is a generalised Reidemeister move induced by a Reidemeister move \( \mathfrak{D}_{i} \rightarrow \mathfrak{D}_{i+1} \) on \( \Sigma \). \Cref{Prop:gparity1} guarantees that \( f_{\mathscr{C}} \) satisfies the axioms of \Cref{Def:parityaxioms} for the move \( \mathfrak{D}_{i} \rightarrow \mathfrak{D}_{i+1} \).
	
	Consider the function induced by \( f_{\mathscr{C}} \) on the virtual link diagrams \( D_j \); it follows immediately from the observations above that this function satisfies the parity axioms for a move \( D_{j} \rightarrow D_{j+1} \).
\end{proof}

Henceforth we shall use \( f_{\mathscr{C}} \) to denote both the parity for link diagrams on \( \Sigma \), and the induced parity for associated virtual link diagrams. We refer to both as the \emph{\( \mathscr{C} \)-parity}.

\subsection{Topological definition}\label{Sec:top}
The \( \mathscr{C} \)-parity enjoys a topological definition in terms of covering spaces that we now describe (for a similar construction in the case of the Gaussian parity for virtual knots, see \cite[Section 5]{Boden2017}). Given a link diagram \( \mathfrak{D} \) on \( \Sigma,\) suppose that \( \gamma \subset \Sigma \) is a simple closed curve with even intersection number with every component of \( \mathfrak{D} \). Let \( \pi : \widetilde{\Sigma}  \times I \to \Sigma \times I \) be a double cover of \( \Sigma \times I \) formed by cutting two copies of \( \Sigma \times I \) along \( \gamma \times I \), and identifying the resulting boundaries (see \Cref{Fig:cover}). The diagram \( \mathfrak{D} \) represents a link, \( \mathfrak{L} \), in \( \Sigma \times I \), and \( {\pi}^{-1} \left( \mathfrak{L} \right) \) is a link with twice as many components as \( \mathfrak{L} \) (this is a consequence of the intersection condition on \( \mathfrak{D} \) and \( \gamma \)).

A \( \gamma \)-colouring of \( \mathfrak{D} \) is equivalent to a choice of lifting of \(\mathfrak{L}\) to a link \( \widetilde{\mathfrak{L}} \) in \(\widetilde{\Sigma} \times I\): the colouring may be used to indicate a choice of preferred component of \( {\pi}^{-1} \left( \mathfrak{L} \right) \) for each component of  \( \mathfrak{L} \) (an example is given in \Cref{Fig:cover}). Specifically, segments of distinct colours are lifted to distinct sheets of the cover \(\pi: \widetilde{\Sigma}\times I \to \Sigma \times I\). A subtlety is presented by the fact that \( {\pi}^{-1} ( \gamma \times I ) \) consists of two connected components. Given a \( \gamma \)-colouring of \( \mathfrak{D} \), the specific intersection between the components of \( \widetilde{\mathfrak{L}} \) and \( {\pi}^{-1} ( \gamma \times I ) \) may be determined by choosing orientations for \(\mathfrak D\) and \( \gamma \), and comparing the sign of the intersection point \(\mathfrak{D} \cap \gamma\) with the colours of \(\mathfrak D\) as it crosses \(\gamma\). The lift does not depend of the orientations for \(\mathfrak D\) and \( \gamma \).

Further, it is clear that such a choice of \( \widetilde{\mathfrak{L}} \) defines a \( \gamma \)-colouring of \( \mathfrak{D} \), via the correspondence between sheets of the cover and colours of the segments. It follows that the parity constructed in \Cref{Sec:comb} may alternatively be defined in terms of a topological choice of lift, as opposed to a combinatorial choice of colouring.

Recall that given a parity theory, parity projection is the removal of odd crossings by converting them to virtual crossings. In the topological setting described here the link \( \widetilde{\mathfrak{L}} \) is precisely the link obtained from \( \mathfrak{L} \) by (the appropriate notion of) parity projection. Specifically, let \(D\) be a virtual link diagram defined by \( \mathfrak{D} \), and \( \mathscr{C} \) the \( \gamma \)-colouring associated to the choice of lift \(  \widetilde{\mathfrak{L}} \). If \( p(D) \) is the diagram obtained from \( D \) by parity projection with respect to \( f_{\mathscr{C}} \), then \( p(D) \) is a diagram of the virtual link represented by \( \widetilde{\mathfrak{L}} \).

As such, it follows that parity projection is realised as the operation of lifting \( \mathfrak{L} \) to a prescribed double cover of \( \Sigma \times I\).

\begin{figure}
	\centering
	\includegraphics[scale=0.5]{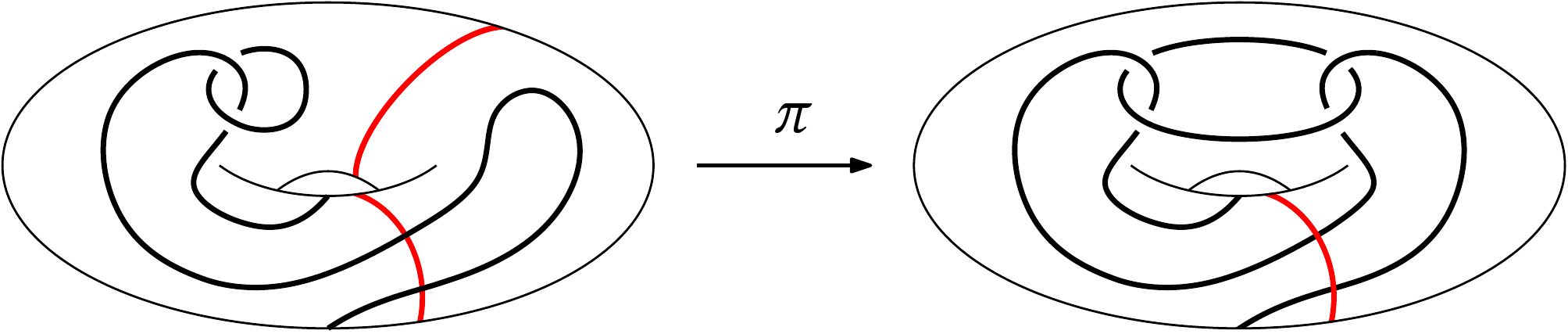}
	\caption{The \( \gamma \)-colouring of \Cref{Ex:gcolour} is equivalent the depicted choice of lift.}\label{Fig:cover}
\end{figure}

\section{Applications of homological parity for links}\label{Sec:apps}
With the \( \mathscr{C} \)-parity in place, we use it to obtain our main result in \Cref{Sec:thm}, before presenting further applications in \Cref{Sec:lem}.

\subsection{Proof of \Cref{Thm:main}}\label{Sec:thm}
We begin by recalling some necessary definitions. 
\begin{definition}[Carter surface \cite{Carter91,Kamada2000}]\label{Def:cs}
	Let \( D \) be a virtual link diagram. Consider \( D \) as an abstract \(4\)-valent graph (in which classical crossings are vertices, virtual crossings are not); there is an orientable surface with boundary, \( F \), that deformation retracts onto this graph. The \emph{Carter surface} of \( D \) is the closed orientable surface obtained by gluing discs to the boundary of \( F \).
\end{definition}
For a virtual link diagram, \( D \), the construction of the Carter surface of \( D \) naturally produces a link diagram \( \mathfrak{D} \) on the Carter surface, such that \( D \) corresponds to \( \mathfrak{D} \). This construction is unaffected by detour moves; the Carter surface associated to a virtual link diagram depends only on its underlying Gauss diagram.

A \emph{subdiagram} of a virtual link diagram \( D \) is a diagram obtained by converting a (possibly empty) subset of the classical crossings of \(D\) to virtual crossings. A \emph{proper subdiagram} is formed by converting a non-empty subset of classical crossings.

\Cref{Thm:main} is a consequence of the following result that enjoys wider utility: applications to ascending number and bridge number are given in \Cref{Sec:lem}.

\begin{theorem}\label{Thm:sub}
	A diagram of a virtual link \(L\) possesses a subdiagram that also represents \(L\) and has minimal genus.
\end{theorem}

\begin{proof}
	Let \( D \) be a diagram of the virtual link \(L\), and \( \Sigma \) the Carter surface of \( D \). Denote by \( \mathfrak{D} \) the diagram on \( \Sigma \) defined by \( D \). If \( \Sigma \) realises the supporting genus of \( L \) then \( D \) is the desired minimal genus subdiagram.
	
	Suppose that \( \Sigma \) does not realise the supporting genus of \( L \). By Kuperberg's Theorem \( \mathfrak{ D } \) is related by a finite sequence of Reidemeister moves on \( \Sigma \) to a diagram \( \mathfrak{D}' \), such that there exists an essential simple closed curve, \( \gamma \subset \Sigma \), disjoint to \( \mathfrak{D}' \). Denote this sequence as
	\begin{equation}\label{Eq:ss}
	\mathfrak{D} = \mathfrak{D}_1 \rightarrow \mathfrak{D}_2 \rightarrow \cdots \rightarrow \mathfrak{D}_n = \mathfrak{D} '.
	\end{equation}
	Without loss of generality, we may assume that \(\gamma\) is such that \Cref{Eq:ss} is the shortest possible sequence of this kind. It follows that \(\mathfrak{D}_i\) is cellularly embedded for \(1\leq i < n\) (that is, the complement of a neighbourhood of \( \mathfrak{D}_i \) in \( \Sigma \) is a disjoint union of discs).
	
	The sequence of \Cref{Eq:ss} defines a sequence of virtual link diagrams related by generalised Reidemeister moves, denoted as
	\begin{equation}\label{Eq:sd}
	D = D_1 \rightarrow D_2 \rightarrow \cdots \rightarrow D_m = D ',
	\end{equation}
	with \( m \geq n \). Our assumption on the length of the sequence of \Cref{Eq:ss} implies that \( \mathfrak{D}_{n-1} \rightarrow \mathfrak{D}_n \) is an \( RII \) move of the form
	\begin{equation*}
	\raisebox{-38pt}{\includegraphics[scale=0.8]{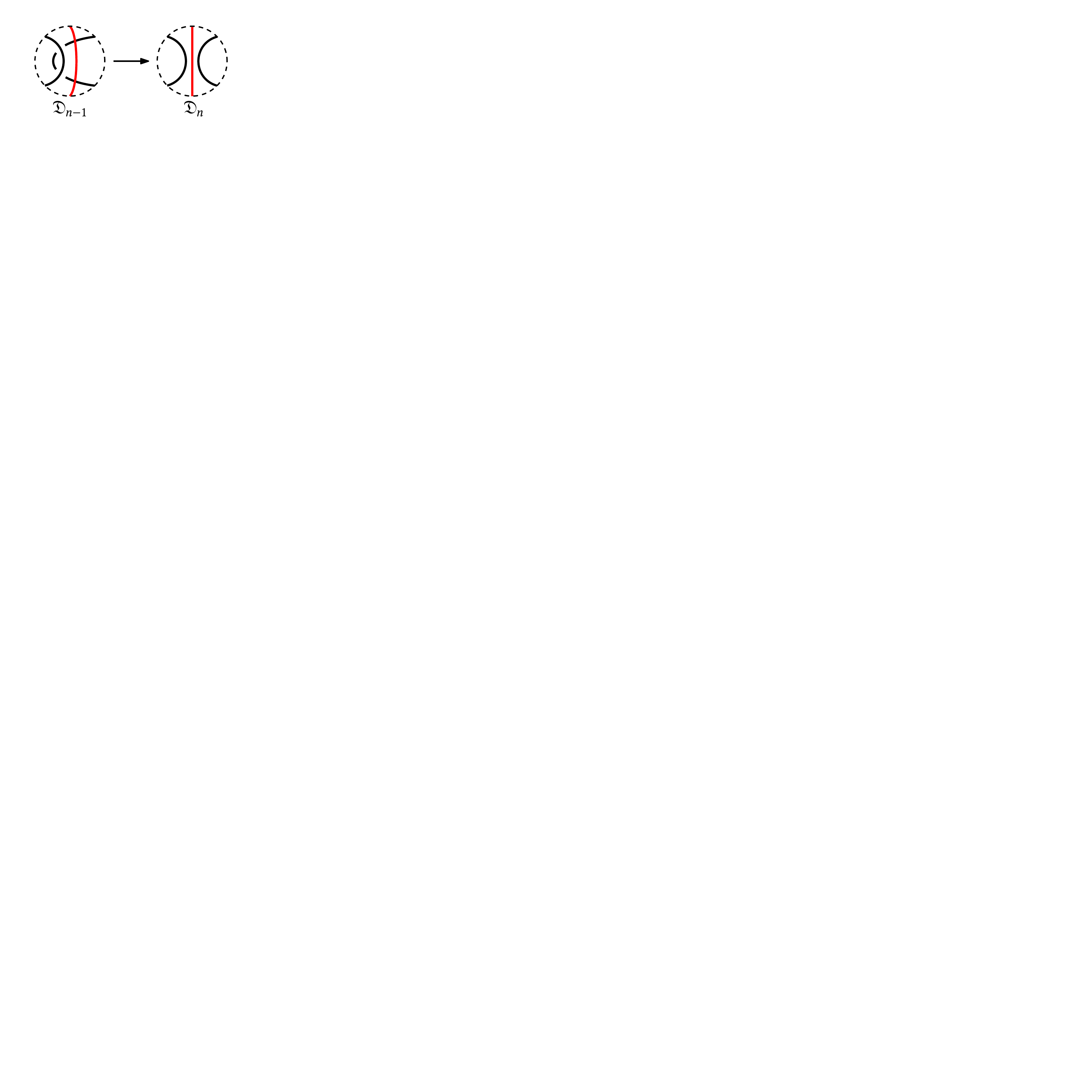}}.
	\end{equation*}
	We may further assume that \( \gamma \) (depicted by the red arc) is disjoint to \( \mathfrak{D}_{n-1} \) outside of the disc neighbourhood depicted above. Observe that \( \mathfrak{D}_{n-1} \) has intersection number \(0\) with \( \gamma \), and that there exists a \( \gamma \)-colouring of \( \mathfrak{D}_{n-1} \) such that all components possess the same colour, except in the region supporting the \(RII\) move, where the colouring is:
	\begin{equation*}
	\raisebox{-28pt}{\includegraphics[scale=0.85]{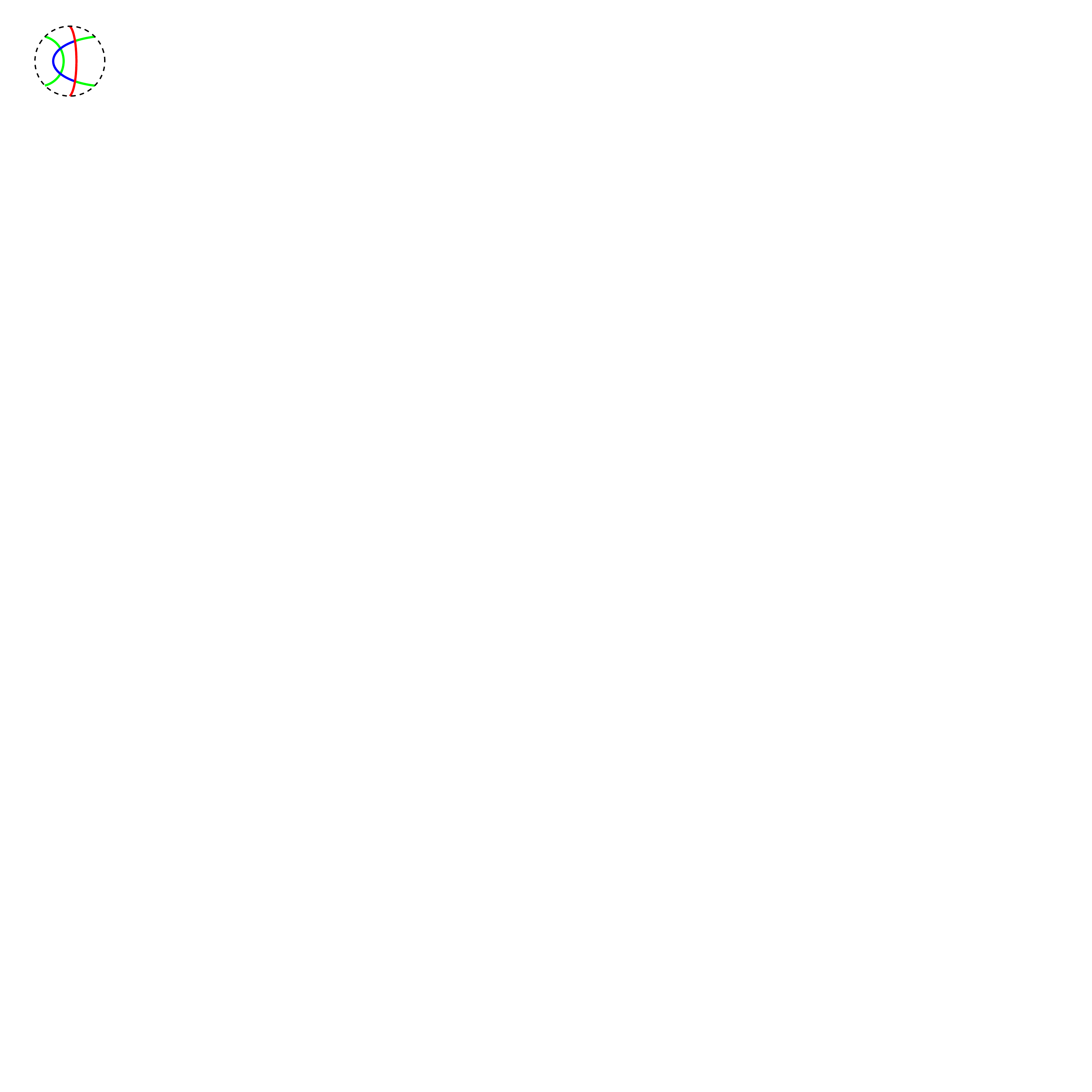}}.
	\end{equation*}
	Denote this distinguished \( \gamma \)-colouring by \( \mathscr{C} \). Every diagram \( \mathfrak{D}_i \) in the sequence possesses a \( \gamma \)-colouring induced by \( \mathscr{C} \); by abuse of notation also denote these \( \gamma \)-colourings by \( \mathscr{C} \). Notice that every crossing of \( \mathfrak{D}_n \) is even with respect to the parity \( f_{\mathscr{C}} \). The same is true for \( D_m \). 
	
	This analysis of the sequence of \Cref{Eq:ss}, combined with \Cref{Prop:gparity2}, allows us to project the sequence of \Cref{Eq:sd} to obtain a new sequence
	\begin{equation}\label{Eq:s2}
	p(D) = p(D_1) \rightarrow p(D_2) \rightarrow \cdots \rightarrow p(D_m) = p(D '),
	\end{equation}
	where \( p (D_i) \) denotes the virtual link diagram obtained from \( D_i \) via parity projection with respect to \( f_{\mathscr{C}} \). Every crossing of \( D_m \) is even with respect to \( f_{\mathscr{C}} \), so that \( p ( D_m ) = D_m \), and we may concatenate the sequence of \Cref{Eq:sd} with that of \Cref{Eq:s2} (in reverse) to obtain a sequence of generalised Reidemeister moves from \( D \) to \( p(D) \). Thus \( D \) and \( p(D) \) both represent the virtual link \(L\).
	
	We claim that the sequence of \Cref{Eq:s2} preserves the Carter surface. That is, that the diagrams \( p(D_i) \) and \( p(D_j) \) have homeomorphic Carter surfaces, for all \( 1 \leq i,j \leq m \). To see this, first recall that the only generalised Reidemeister move that may possibly alter the homeomorphism type of the Carter surface is an \( RII \) move. Suppose that \( p ( D_{i} ) \rightarrow p ( D_{i+1} ) \) is an \( RII \) move; then \( D_{i} \rightarrow D_{i+1} \) is an \( RII \) move involving even crossings (with respect to \( f_{\mathscr{C}} \)).
	
	To verify that \( p ( D_{i} ) \rightarrow p ( D_{i+1} ) \) preserves the Carter surface we employ the topological definition of the \( \mathscr{C} \)-parity. The move \( D_{i} \rightarrow D_{i+1} \) is associated to a move \( \mathfrak{D}_{j} \rightarrow \mathfrak{D}_{j+1} \) of \Cref{Eq:ss}. By assumption this latter move occurs on a disc neighbourhood of \( \Sigma \), and \( \mathfrak{D}_{j} \), \( \mathfrak{D}_{j+1} \) are cellularly embedded. It follows that there is a disc, \( \Delta \), as indicated in \Cref{Fig:oddr2}. Suppose that \( \pi : \widetilde{\Sigma} \to \Sigma \) is the double cover associated to \( \gamma \) (as described in \cref{Sec:top}); one component of \( {\pi}^{-1} ( \Delta ) \) must be as depicted in \Cref{Fig:oddr2}.
	
	Let \( \widetilde{\mathfrak{D}}_{j} \), \( \widetilde{\mathfrak{D}}_{j+1} \) denote the lifts of \(\mathfrak{D}_{j} \), \( \mathfrak{D}_{j+1} \) prescribed by \( \mathscr{C} \) (the diagrams \( \widetilde{\mathfrak{D}}_{j} \), \( \widetilde{\mathfrak{D}}_{j+1} \) are diagrams on \( \widetilde{\Sigma} \)). As described in \Cref{Sec:top}, lifting with respect to \( \pi \) is equivalent to parity projection with respect to \( f_{\mathscr{C}} \). In particular, \( \widetilde{\mathfrak{D}}_{j} \) and \( p ( D_{i} ) \) have the same Gauss diagram, as do \( \widetilde{\mathfrak{D}}_{j+1} \) and \( p ( D_{i+1} ) \). It follows that we may obtain the Carter surface of \( p ( D_{i} ) \) by cutting out a neighbourhood of \( \widetilde{\mathfrak{D}}_{j}  \) in \( \widetilde{\Sigma} \), and capping its boundary with discs; the Carter surface of \( p ( D_{i+1} ) \) is obtained in the same manner from \( \widetilde{\mathfrak{D}}_{j+1}  \).
	
	Depending on whether the \( RII \) move \( \widetilde{\mathfrak{D}}_{j} \rightarrow \widetilde{\mathfrak{D}}_{j+1} \) adds or removes crossings, one of the diagrams involved is as the top-right diagram in \Cref{Fig:oddr2}. Notice that the arcs of this diagram involved in the \( RII \) move are part of the boundary of a single disc component of \( {\pi}^{-1} ( \Delta ) \). It follows that \( p ( D_{i} ) \) and \( p ( D_{i+1} ) \) have homeomorphic Carter surfaces, and the sequence of \Cref{Eq:s2} preserves the Carter surface.
	
	\begin{figure}
		\includegraphics[scale=0.5]{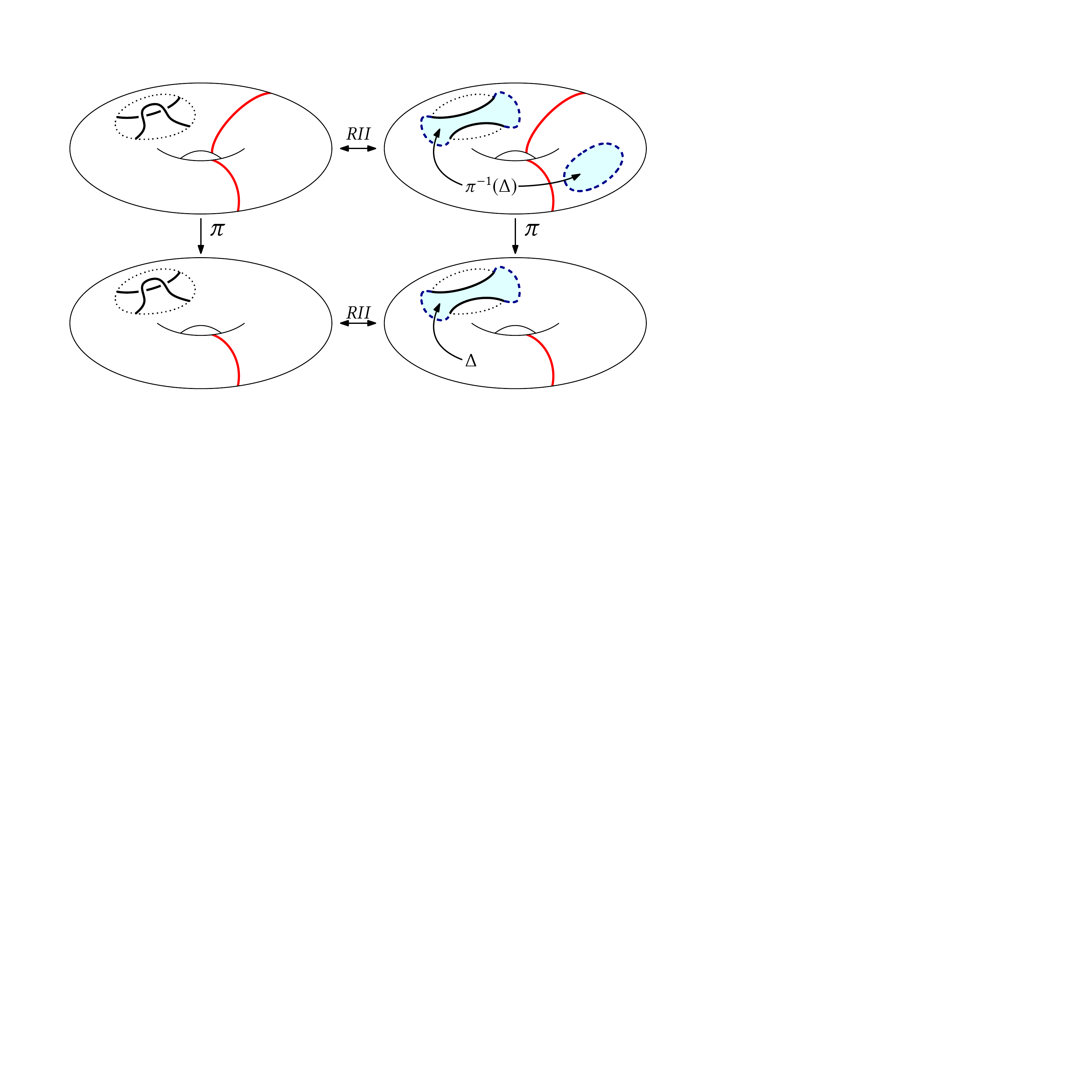}
		\caption{Lifting an \( RII \) move involving even crossings, using the cover associated to \( \gamma \).}
		\label{Fig:oddr2}
	\end{figure}
	
	Next, we claim that \( D \) contains an odd crossing with respect to \( f_{\mathscr{C}} \), so that \( p(D) \) is a proper subdiagram of \(D\) that represents \(L\).
	
	Assume towards a contradiction that \( D \) does not contain an odd crossing with respect to \( f_{\mathscr{C}} \). Then \( p(D) = D \), so that \( p(D) \) has Carter surface \( \Sigma \) also. As \( p(D_m) = D_m \) the diagram \( p(D_m) \) has Carter surface \( \Sigma ' \), obtained from \( \Sigma \) by destabilizing along \( \gamma \). This destabilization must change the homeomorphism type of \( \Sigma \): if \( \gamma \) is non-separating then \( \Sigma \) and \( \Sigma ' \) have different genera, and if \( \gamma \) is separating then \( \Sigma \) and \( \Sigma ' \) have a different number of connected components. It follows that the Carter surfaces of \( p(D) \) and \( p(D_m) \) are not homeomorphic. But the sequence of \Cref{Eq:s2} (starting at \( p(D) \) and ending at \( p(D_m) \)) preserves the homeomorphism type of the Carter surface, hence the desired contradiction.
	
	In conclusion, we have produced \( p(D) \) that represents \( L \) and is a proper subdiagram of \( D \). If \( p(D)\) is not a minimal genus diagram of \( L \), then we may repeat the process described above (with a different curve in place of \( \gamma \), guaranteed to exist by Kuperberg's Theorem). The proof is completed by iterating this process: after a finite number of iterations a minimal genus subdiagram of \( D \) is obtained, that must also represent \(L\) (the number of iterations required is bounded above by the number of classical crossings of \( D \)). 
\end{proof}

\begin{proof}[Proof of \Cref{Thm:main}]
	Suppose that \( D \) realises the crossing number of \( L \). Apply \Cref{Thm:sub} to obtain a new diagram of \(L\), denoted \( D' \). It is guaranteed that \( D' \) is minimal genus and a subdiagram of \( D \). By hypothesis \( D \) is of minimal crossing number, hence \( D = D' \), so that \( D \) is a minimal genus diagram.
\end{proof}

\subsection{Realising the bridge and ascending numbers on minimal genus diagrams}\label{Sec:lem}
We present further examples of the utility of \Cref{Thm:sub}. First, we use it to prove that the bridge number of a virtual link is realised on a minimal genus diagram. This extends the corresponding result in the case of virtual knots due to Chernov \cite{Chernov2013} and Manturov \cite{Manturov13}. 

A \emph{bridge} of a virtual link diagram is an arc that contains one or more overcrossings (and any number of virtual crossings). The \emph{bridge number} of a link \(L\) is the minimum number of bridges over all diagrams for \(L\) (for further details see \cite{Hirasawa2011}). It is \emph{a priori} conceivable that a classical link may admit a virtual link diagram with fewer bridges than any of its classical diagrams: we use \Cref{Thm:sub} to show that this cannot occur.

\begin{proposition}\label{Prop:bridge}
	The bridge number of a virtual link is realised on a minimal genus diagram. In particular, the bridge number of a classical link does not decrease when it is considered as a virtual link.
\end{proposition}

\begin{proof}
	Let \( D \) be a diagram that realises the bridge number of a virtual link \( L \). Apply \Cref{Thm:sub} to \( D \) to produce a minimal genus subdiagram of \( D \), that also represents \( L \), and observe that changing classical crossings to virtual crossings cannot increase the bridge number of a diagram.
\end{proof}

The meridional rank conjecture \cite[Problem 1.11]{Kirby} posits that the bridge number of a classical link is equal to the meridional rank of the link group. An affirmative answer to this conjecture would provide an alternative proof of the fact that the bridge number of a classical link does not decrease when it is considered as a virtual link (as established in \Cref{Prop:bridge}). In fact, the meridional rank conjecture implies the stronger statement that the bridge number of a classical link does not decrease when it is considered as a \emph{welded} link. (For further details see \cite{Boden2015}.)

Next, we consider the ascending number, a numerical invariant introduced by Ozawa \cite{Ozawa2010} (also known as the \emph{warping degree} \cite{Shimizu2011}). The definition readily extends to virtual links, as follows. An oriented virtual knot diagram is said to be \emph{ascending} if one encounters only undercrossings, or crossings that have previously been met, when traversing the diagram from an arbitrary basepoint. There is a similar definition for oriented virtual link diagrams (given basepoints on and an ordering of the link components). The ascending number of an oriented virtual link diagram \( D \), denoted \( a(D) \), is the minimum number of crossing changes needed to make it ascending. The ascending number of an oriented virtual link is the minimal ascending number of a diagram representing the link.

\Cref{Thm:sub} applies to show that the ascending number of a classical link is preserved when passing to the virtual category.  
\begin{proposition}\label{Prop:asc}
The ascending number of a virtual link is realised on a minimal genus diagram. The ascending number of a classical link does not decrease when it is considered as a virtual link.
\end{proposition}
 
\begin{proof}
	The claim follows easily from \Cref{Thm:sub}: notice that if \(D\) is a virtual link diagram with subdiagram \(D'\) then \(a(D') \leq a(D).\)
\end{proof}

We conclude this note by advertising the following open question. The \emph{unknotting number} of a classical or virtual knot \( K \) is the minimal number of crossing changes \( \raisebox{-4pt}{\includegraphics[scale=0.35]{cc1.pdf}} \to \raisebox{-4pt}{\includegraphics[scale=0.35]{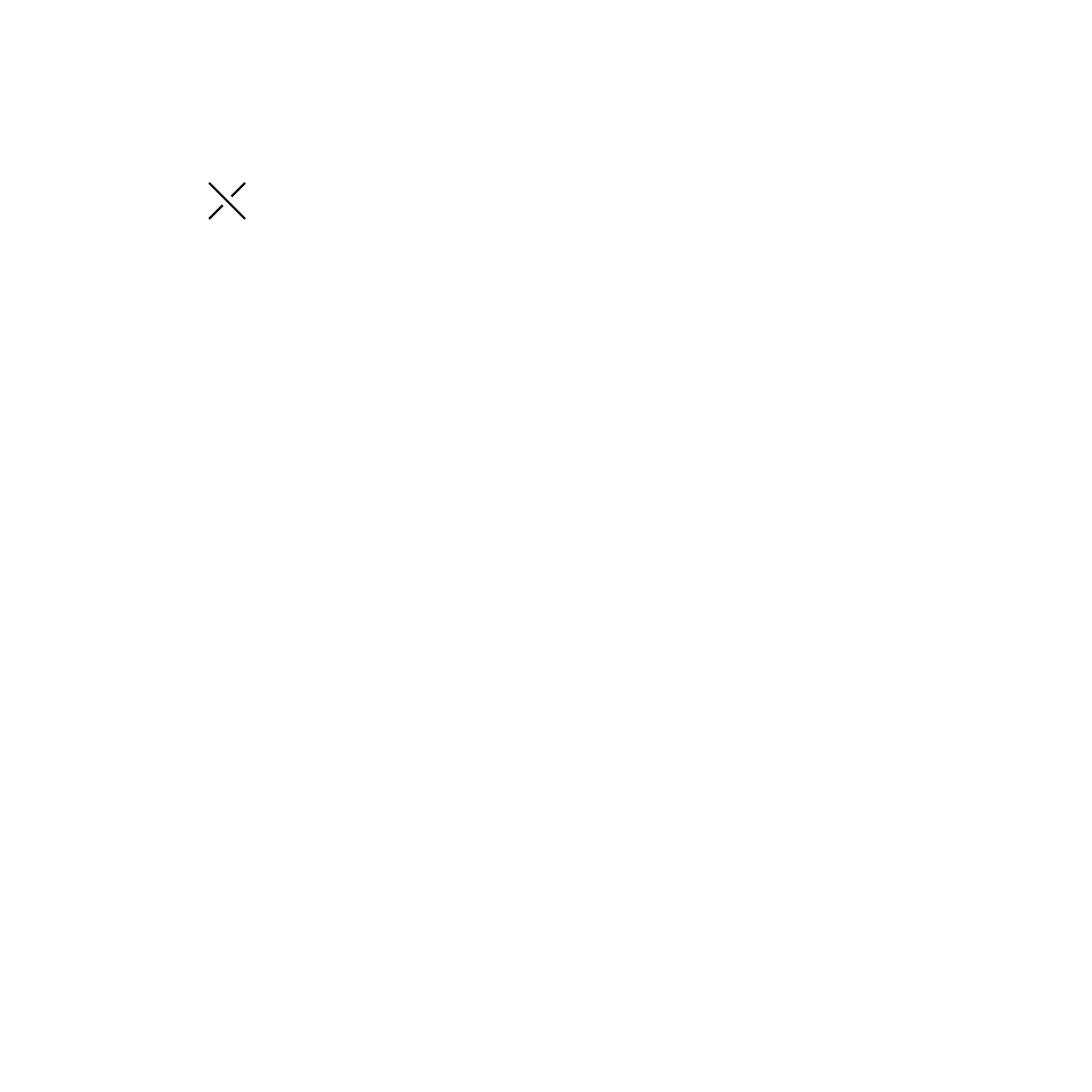}} \) needed to convert a diagram of \( K \) to an unknot diagram. The \emph{unlinking number} of a virtual link \( L \) is defined similarly. Not all virtual knots (links) can be unknotted (unlinked) by crossing change, in which case the unknotting (unlinking) number is defined to be infinite.

\begin{question}
	Is the unknotting number of a classical knot preserved when it is considered as a virtual knot? Is the same true for the unlinking number of a classical link? For virtual knots and links, are the unknotting and unlinking numbers attained on minimal genus diagrams?
\end{question}
The unknotting number of a classical knot is bounded above by its ascending number \cite{Ozawa2010}, so that \Cref{Prop:asc} provides evidence in favour of a positive answer to the classical cases of the question posed above.

Another interesting open question is obtained by replacing `unlinking' with `splitting'. (Again, not all virtual links can be split by crossing change, in which case the splitting number is defined to be infinite.)

Finally, we note the operation of converting classical crossings to virtual crossings is an unknotting, unlinking, and splitting operation for virtual links. It is interesting to consider questions analogous to those above in the context of this operation.
 
\bibliographystyle{plain}
\bibliography{library}

\end{document}